\documentclass[11pt]{amsart}
\usepackage{amsmath, amsthm, amscd, amsfonts, amssymb, graphicx, color, xypic}

\newtheorem{theorem}{Theorem}[section]

\theoremstyle{definition}
\newtheorem{definition}[theorem]{Definition}
\newtheorem{example}[theorem]{Example}

\theoremstyle{remark}
\newtheorem{remark}[theorem]{Remark}
\numberwithin{equation}{section}

\begin{document}

\title[Product spaces generated by bilinear maps]
{Product spaces generated by bilinear maps and duality}
 
\author[E.A. S\'{a}nchez P\'{e}rez]{
 E.A. S\'{a}nchez P\'{e}rez}

\address{ 
\newline
Instituto Universitario de Matem\'atica Pura y Aplicada
\newline
Universitat Polit\`ecnica de Val\`encia
\newline
Camino de Vera s/n, 46022
Valencia, Spain.}
\email{\textcolor[rgb]{0.00,0.00,0.84}{easancpe@mat.upv.es}}



\thanks{Support of the Ministerio de Econom\'{\i}a y Competitividad (Spain) MTM2012-36740-C02-02.}

\begin{abstract}
We analyze a definition of product of Banach spaces that is naturally associated by duality with an abstract notion of space of multiplication operators. This dual relation allows to understand several constructions coming from different fields of the functional analysis, that can be seen as instances of the abstract  one when a particular product is considered. Some relevant examples and applications are shown.
\end{abstract}

\maketitle

\section{Introduction and notation}

One of the main tools in the theory of operator ideals on Banach spaces is the so called
representation formula for maximal operator ideals. This result asserts that if $E,F$ are Banach 
spaces, $F^*$ is the dual of $F$ and $\mathcal U$ is an operator ideal, we can find 
a reasonable tensor norm 
$\alpha$ such that
$(E \otimes_\alpha F)^*=\mathcal{U}(E,F^*)$
(see for example \cite[Ch.II,\S 17]{deflo}). For the
particular case of $\alpha=\pi$ ---the projective tensor norm---
we have the well-known representation of the ideal of linear and
continuous operators,
$$
(E \otimes_\pi F)^*=\mathcal{L}(E,F^*).
$$

 Consider now the following case, that comes from a different setting. Take a couple of Banach function spaces $X(\mu)$ and $Y(\mu)$ over a finite measure $\mu$, and consider the  K\"othe dual $Y(\mu)'$ of $Y(\mu)$ ---the elements of the dual space that can be represented as integrals---. Assume also that the so called product space $X(\mu) \pi Y(\mu)$ is again a Banach function space; the saturation requirements on the product  for this to hold are well-known.  Then we have that
$$
(X(\mu) \pi Y(\mu))'= X(\mu)^{Y(\mu)'},
$$
 where $X(\mu)^{Y(\mu)'}$ is the space of multiplication operators from $X(\mu)$ to the K\"othe dual $Y(\mu)'$ (see Proposition 2.2 in \cite{delsan}; see also \cite{caldelsan,kolesmal,SPfactomulti,suktom,schepmultip}).

Other example coming in this case from the harmonic analysis is given by the Hadamard product; this case has been investigated in \cite{blaspav} and will be analyzed in Section 3.
The duality formulas that are known for the $\Sigma$ and $\Delta$ constructions  coming from the interpolation theory for Banach spaces can be understood in a sense from the same point of view (see \cite[Section 3]{massan} and the references therein).
These examples  suggest that this general formula would make sense using an abstract definition of product space and taking into account the associated duality relation and the corresponding space of linear and continuous operators that fit well with the rest of the elements. There are much more examples that can be adapted to this general scheme and we do not treat in this paper: for instance, the convolution product on the class of the $L^p$-spaces, the pointwise product on $C(K)$ spaces or the composition product in an abstract Banach algebra. All of them may be adapted to make sense in our general framework, providing new representations for the corresponding ``dual" spaces. 

If $E$ and $F$ are Banach spaces,  the projective tensor product $E \otimes_\pi F$ is defined to be the linear space of all finite combination of single tensors $x \otimes y$ together with the norm
$$
\pi(z):= \inf \sum_{i=1}^n \|x_i \| \| y_i \|,
$$
where the infimum is computed over all finite representations of
$z$ as $\sum_{i=1}^n x_i \otimes y_i$. If $c:E \times F \to G$ is a Banach valued bilinear map, the range of its linearization $\hat{c}$ from $E \otimes_\pi F$ to $G$ can be used to define a sort of product structure that can be identified with a subspace of $G$ with a particular norm. Together with some subspaces of operators that we call generalized dual spaces, it allows to establish the duality formula that is the objective of this work. It must be noted that our aim is to understand a lot of classical results from a unified point of view and not to prove a genuine ``new" result; the proof of the only theorem of the paper is easy and in a sense standard. How to write results coming from several settings that seem to be completely different  as  consequences of a unified principle is what we want to show.

In this paper we consider in a sense linear and continuous operators as generalized multiplication operators. This is the reason that we use the usual notation for multiplication operators, that in a sense is the opposite to the  one for maps: if $Y$ and $X$ are Banach spaces we will write $Y^X$ for the space of linear and continuous operators from $Y$ to $X$, that is usually denoted by $\mathcal{L}(Y,X)$. This classical notation will be used too, depending on the context.

\section{Generalized duality on Banach spaces and the product duality  formula}

In this section we introduce the construction that leads to what we call a product of two Banach spaces and the associated product duality formula  in terms of bilinear maps and ``dual" subspaces of operators.

 Let us fix some notation. Let  $c:E \times F \to G$ be  a bounded bilinear map and let $\mathcal{V} \subseteq \mathcal{L}(G,X)$.

\begin{itemize}

\vspace{2mm}

\item
 We write $\hat{c}$ for the linearization of
$c$, i.e.
$\hat{c}: E {\otimes} F \to G$, $\hat{c}(x \otimes y):=c(x,y),$ $x \in E$, $y \in F.$

\vspace{4mm}

\item We  write $c_L$ ---the left linearization of $c$---
for the map $c_L:E \to \mathcal{L}(F,G)$ given by
$c_L(x)(y):=c(x,y)$, $x \in E$, $y \in F$.

\vspace{4mm}

\item We  define
$c_R$ ---the right linearization of $c$--- as the map $c_R:F \to \mathcal{L}(E,G)$ given by
$c_R(y)(x):=c(x,y)$, $x \in E$, $y \in F$.

\end{itemize}

\vspace{4mm}

Let us introduce now our basic product structure.

\begin{definition}
Consider a pair of normed spaces $E$ and $F$, and assume that
there is a continuous Banach space valued bilinear map $c:E \times
F \to G$. Define a seminorm on $E \otimes F$ by the formula
$$
\pi_{c}(z):= \inf \big\{ \sum_{i=1}^n \|x_i\| \|y_i\|:
\hat{c}(z)=\hat{c}(\sum_{i=1}^n x_i \otimes y_i)= \sum_{i=1}^n {c}(x_i,y_i) \big\}, \quad z \in E \otimes F,
$$
where the infimum is defined for all simple tensors $\sum_{i=1}^n x_i \otimes y_i$ such that $\hat{c}(z)=\hat{c}(\sum_{i=1}^n x_i \otimes y_i).$
Note that it is a \textit{norm} if we construct a quotient by identifying the equivalence classes  of the tensor product $E \otimes F$ with its range by $\hat{c}$ in $G$, i.e. with the subspace of $\hat{c}(E \otimes_{\pi} F) \subseteq G.$
We define the \textit{product space} $E {\otimes}_{\pi_{c}} F$ as  the normed space $(\hat{c}(E \otimes F), \pi_c )$.
\end{definition}

The following scheme may help to visualize the construction above: the product space is a sort of intermediate space in the commutative diagram
$$
\xymatrix{ E \times F \ar[rr]^{c} \ar@{->}[dr]_{\otimes} & & E
{\otimes}_{\pi_{c}} F  \hookrightarrow  G. \\
&  E {\otimes}_{\pi} F \ar[ur]_{\hat{c}} & }
$$
We will use the abuse of notation of considering $\hat{c}$ having values both in $G$ and in $ E
{\otimes}_{\pi_{c}} F$ depending on the context.
\vspace{3mm}

\begin{remark} \label{defpro}
Let  $\mathcal{V}$ be a Banach subspace  of
the space of linear and continuous operators $\mathcal{L}(G,X)$.
Fix $T \in \mathcal{V}$. It defines a diagram as
$$
E \otimes_\pi F \to^{\hat{c}} E {\otimes}_{\pi_{c}} F \hookrightarrow
G \to^T X.
$$
Let us define an operator $\varphi_T:E {\otimes}_{\pi_{c}} F \to X$ by $\varphi_T( \hat{c}(x \otimes y)):= (T \circ c)(x,y),$ $x \in E,$ $y \in F$, and impossing its linearity. It  actually gives  a well-defined continuous operator
from $E {\otimes}_{\pi_{c}} F$ to $X$. To see this, note that for each $x \in E$ and $y \in F$,
$$
\| \varphi_T( \hat{c}(x \otimes y))\| \le \|T\| \|c(x,y)\|,
$$
and so for each finite sum $z=\sum_{i=1}^n x_i \otimes y_i$,
$$
\| \varphi_T( \hat{c}(z))\| = \| \varphi_T( \sum_{i=1}^n c(x_i,y_i))\| \le
\|T(\sum_{i=1}^n c(x_i,y_i))\| \le
\|T\| \|c\| \pi (z).
$$
Since if $\sum_{i=1}^{n_1} c(x^1_i,y^1_i)= \sum_{i=1}^{n_2} c(x^2_i,y^2_i)$ we have that
$$
\varphi_T( \hat{c}( \sum_{i=1}^{n_1} x^1_i \otimes y^1_i))= T(\sum_{i=1}^{n_1} c(x^1_i,y^1_i))=
T(\sum_{i=1}^{n_2} c(x^2_i,y^2_i)) =
\varphi_T(\hat{c} ( \sum_{i=1}^{n_2} x^2_i \otimes y^2_i)),
$$
we also obtain that
$$
\| \varphi_T( \hat{c}(z))\|  \le
\|T\| \|c\| \pi_c (z).
$$
\end{remark}

\vspace{3mm}

This observation is on the basis of the next definition.

\begin{definition}
With the notation used above, consider a closed subspace
$\mathcal{V}$ of operators from $G$ to $X$.  We define the
\textit{(generalized) dual space} $(E {\otimes}_{\pi_c} F)^\mathcal{V}$ of $E {\otimes}_{\pi_c} F$
as
$$
(E {\otimes}_{\pi_c} F)^\mathcal{V} := \{\varphi_T:E {\otimes}_{\pi_c} F \to X \, | \,T \in
 \mathcal{V} \} \subseteq \big(
E {\otimes}_{\pi_c} F \big)^X
$$
endowed with the natural norm $\|\varphi_T\|:=\sup_{x \in B_E, y \in
B_F} \|T(c(x,y))\|_X$.
\end{definition}

\begin{definition}
With the notation used above, consider a closed subspace
$\mathcal{V}$ of operators from $G$ to $X$. For each $T \in
\mathcal{V}$ define the \textit{(generalized pointwise) multiplication
operator} from $E$ to $F^X$ by
$$
S_T(x)(y):=T(c_L(x)(y))= T(c(x,y)) \in X.
$$
We will call them simply multiplication operators for short.
 We will write $E^{\mathcal{V} \odot c_L}$ for the normed space of (generalized pointwise)
multiplication operators
$$
  E^{\mathcal{V} \odot c_L} :=\big\{ S_T:E \to \mathcal{L}(F,X) \,| \, T \in \mathcal{V}
\big\} \subseteq   E^{(F^X)}
$$
endowed with the operator norm of $ E^{\mathcal{L}(F,X)}$. Clearly, $\|S_T\| \le \|T\| \|c\|$ for all $T \in \mathcal{V}.$
\end{definition}
Note that the same kind of definition can be done by changing $c_L$ by $c_R$. Also the dual version of the following result ---the product duality formula with $c_R$ instead of $c_L$---, may be obtained.

\begin{theorem} \label{prop1}
Let $E,F,G$ and $X$ be Banach spaces, and suppose that a bilinear map
$c: E \times F \to G$ and a class of operators  $\mathcal V$ from
$G$ to $X$ are given. Then
$$
\big( \big( E{\otimes}_{\pi_c} F \big)^{X} \supseteq \big) \,\,\,\,\,\,\,\,
(E {\otimes}_{\pi_c} F)^\mathcal{V} =E^{\mathcal{V} \odot c_L}  \,\,\,\,\,\,\,\,
\big( \subseteq E^{(F^X)} \big).
$$
isometrically.
\end{theorem}
\begin{proof}
Recall that each element of $(E {\otimes}_{\pi_c} F)^\mathcal{V}$ is given
by the following expression; if $x \in E,$ $y \in F$ and $T \in
\mathcal{V}$,
$$
\varphi_T( \hat{c}(\sum_{i=1}^n x_i \otimes y_i)):= \sum_{i=1}^n (T
\circ c)(x_i,y_i).
$$

Clearly, such a continuous operator can also be understood as a
continuous linear map $S_T: E \to \mathcal{L}(F,X)$ that is
defined by
$$
S_T(x)(y):=(T \circ c)(x,y) \in X
$$
for $x \in E$ and $y \in F$. It is continuous since
$$
\|S_T(x)(y)\|= \|T(c(x,y))\|  \le \|T\| \|c\|\|x\|\|y\|.
$$
Conversely, the definition of an element $S_T$ of $E^{\mathcal{V} \odot c_L}$ allows to define the corresponding $\varphi_T$, in this case by the formula
$$
\varphi_T( \hat{c}(\sum_{i=1}^n x_i \otimes y_i)):= \sum_{i=1}^n S_T(x_i)(y_i) = T(\sum_{i=1}^n c(x_i,y_i)),
$$
where $\hat{c}(\sum_{i=1}^n x_i \otimes y_i) \in E {\otimes}_{\pi_c} F.$ Then
$$
\| \varphi_T(\hat{c}(\sum_{i=1}^n x_i \otimes y_i) ) \| \le \|T\| \|c\| \big( \sum_{i=1}^n \|x_i\| \|y_i\| \big).
$$
A computation as in Remark \ref{defpro} shows that
in fact
$$
\| \varphi_T(\hat{c}(\sum_{i=1}^n x_i \otimes y_i) ) \| \le \|T\| \|c\| \pi_c (\sum_{i=1}^n x_i \otimes y_i) .
$$
Let us show  the equality of the norms of $\varphi_T$ and $S_T$ for a fixed $T \in \nu$. Consider first an element $z \in E{\otimes}_{\pi_c} F$ such that $\pi_c(z) \le 1$. Fix $\varepsilon >0$ and take a tensor $\sum_{i=1}^n x_i \otimes y_i$ such that $\hat{c}(\sum_{i=1}^n x_i \otimes y_i)=z$ and $\sum_{i=1}^n \|x_i\| \|y_i\| < 1 + \varepsilon$. Then
$$
\|\varphi_T(z)\| = \|
\varphi_T(\hat{c}(\sum_{i=1}^n x_i \otimes y_i)) \| =  \| T(\sum_{i=1}^n c(x_i,y_i)) \|
$$
$$
= \| \sum_{i=1}^n S_T(x_i)(y_i) \| \le \|S_T\| \Big( \sum_{i=1}^n \|x_i\| \|y_i\| \Big) \le \|S_T\| (1+ \varepsilon).
$$
Thus, $ \|\varphi_T \| \le \|S_T\|.$ For the converse, take $\varepsilon>0$ and  norm one elements $x \in E$ and $y \in F$ such that $\|S_T\| \le \|S_T(x)(y)\| + \varepsilon.$ Then
$$
\|S_T\| \le \|T(c(x,y))\| + \varepsilon \le \| \varphi_T( \hat{c}(x \otimes y))\| + \varepsilon
$$
$$
\le \|\varphi_T\| \pi_c(\hat{c}(x \otimes y)) + \varepsilon \le \|\varphi_T\| \|x \| \| y\| + \varepsilon \le \| \varphi_T\| + \varepsilon.
$$
This proves that $\|S_T\| \le \|\varphi_T\|,$ and so $\|S_T\| = \|\varphi_T\|.$

\end{proof}

\section{Examples and applications}

\subsection{Some direct examples}
Let us show first some classical easy examples in which  the product duality formula appears in a natural way.

\begin{example}
Let us present an elementary  example that shows how the topological dual
$F^*$ space of a Banach space $F$ can be interpreted in our
setting. Let $E=\mathbb R$ and $F$ a Banach space. Let $c: \mathbb
R \times F \to G=F$ given by the product $c(r,y)=ry$ and consider
$X= \mathbb R$. Let $\mathcal V= F^{\mathbb{R}}=F^*$.
Then we clearly have that $\mathbb{R} {\otimes}_{\pi_c} F= F$ and
$$
F^*= (\mathbb R {\otimes}_{\pi_c} F)^\mathcal{V}  \subseteq (\mathbb{R} {\otimes}_{\pi_c} F)^{\mathbb{R}}= F^*.
$$
On the other hand, ${\mathbb R}^{F^*}=F^*$, and so we
have the desired equality
$$
(\mathbb R {\otimes}_{\pi_c} F)^\mathcal{V}  =
F^*=(\mathbb{R} {\otimes}_{\pi_c} F)^{\mathbb{R}}=
\mathbb{R}^{F^*}={\mathbb R}^{\mathcal{V} \odot c_L}.
$$
\end{example}

\begin{example} \label{ex4.2}
Let us see that the duality between the dual space of the
projective tensor product $E {\otimes}_\pi F$ and the space of
operators from $E$ to $F^*$ is a particular case of this formula.
Consider the bilinear map $c=\otimes: E \times F \to E \otimes_\pi
F$ and $X= \mathbb R$. Consider the space of all linear and
continuous functionals $\mathcal V= \mathcal{L}(E \otimes_\pi F,
\mathbb R)=(E \otimes_\pi F)^*.$ Then
$$
(E {\otimes}_{\pi_c} F)^\mathcal{V}  =
\big( E {\otimes}_{\pi_c} F \big)^*= E^{\mathcal{L}(F,X)}=E^{F^*}= E^{\nu \odot  c_L}.
$$
\end{example}

\begin{example} \label{exBFS}
Let us show an example involving $p$-th powers of Banach function spaces (see Section 2.2 in \cite{libro}, see also \cite[p.51]{lint}). We need to introduce first some basic notions on Banach function spaces. Let $(\Omega,\Sigma,\mu)$ be a complete finite measure space.
We follow the definition of Banach function space given in  \cite[Def.1.b.17,
p.28]{lint}. A real Banach space
$X(\mu)$ of (equivalence classes of) $\mu$-measurable functions is a Banach function space  over $\mu$ ---also called a K\"othe
function space--- if $X(\mu) \subset L^1( \mu)$ and
contains all the simple functions and, if
 $\Vert\cdot\Vert_{X(\mu)}$ is the norm of the space, $g\in X(\mu)$ and $f$ is a measurable function such that  $|f|\le|g|$ $\mu$--a.e., then $f\in
X(\mu)$ and $\Vert f\Vert_{X(\mu)}\le\Vert g\Vert_{X(\mu)}$. The relations
 $L^\infty(\mu)\subset X(\mu)\subset L^1(\mu)$ with continuous inclusions always hold. A
Banach function space\ $X$ is \emph{order continuous} if
decreasing positive  sequences  converging  $\mu$-a.e. to $0$ converge also in the
norm. If $X(\mu)$ and $Y(\mu)$ are Banach function spaces and $X(\mu)
\subseteq Y(\mu)$, we define the \textit{space of multiplication
operators} $X(\mu)^{Y(\mu)}$ as the space of (classes of) measurable functions
defining operators from $X(\mu)$ to $Y(\mu)$ by pointwise multiplication. The operator norm is considered for this space; then $X(\mu)^{Y(\mu)}$  is also a Banach function space
over $\mu$. The {\it K\"othe dual} of $X(\mu)$ is defined by the real functionals obtained by integrating the evaluation of the elements of  $X(\mu)^{L^1(\mu)}$; i.e.
each  $g\in X(\mu)'$ can be identified with a continuous
functional on $X(\mu)$ via the integral of the multiplication
operator $f\in X(\mu)\mapsto fg\in L^1(\mu)$. If $X(\mu)$ is order continuous, then $X(\mu)'=X(\mu)^*$ isometrically. 

Let us define now the notion of $p$-th power of a Banach function space $X(\mu)$. Let $0<p\le 1$.
 The $p$-th power  of $X(\mu)$  is
defined as the set of functions
$$
X_{[p]}:=\{f \in L^0(\mu): |f|^{1/p} \in X(\mu)\}
$$
that is a Banach function space over $\mu$  with
the norm $\|f\|_{X(\mu)_{[p]}}:= \| |f|^{1/p} \|_{X(\mu)}^p$, $f \in
X(\mu)_{[p]}$. For example, $L^1[0,1]_{[1/2]} = L^2[0,1].$

 Let $X(\mu)$ be  an order continuous Banach function space over a finite measure $\mu$. 
Consider the case
$E=X(\mu)_{[1/p]}$, $F=X(\mu)_{[1/p']}$, $G=X(\mu)$ and $X=\mathbb R$. Let $c:X(\mu)_{[1/p]}
\times X(\mu)_{[1/p']} \to X(\mu)$ given by the pointwise product $c(f,g)=f g.$ Then it is well-known that
$X(\mu)$ can be written as the (pointwise) product space $X(\mu)_{[1/p]}
\cdot X(\mu)_{[1/p']}$ endowed with a norm given by the expression
$$
\|f\| := \inf \,\,\, \|g\|_{X(\mu)_{[1/p]}} \cdot \|h\|_{X(\mu)_{[1/p']}} = \inf \,\,\, \| |g|^p\|_{X(\mu)}^{1/p} \cdot \| |h|^{p'}\|_{X(\mu)}^{1/p'},
$$
where the infimum is computed over all decompositions of $|f|$ as  $|f|=|g| |h|$, $ g \in X(\mu)_{[1/p]}$, $h \in X(\mu)_{[1/p']}$ (to see this, use for example Proposition 1.d.2 in \cite{lint}).

 Take $\mathcal V= X(\mu)'$, and consider the space of multiplication operators $\mathcal{M}(X(\mu)_{[1/p]},{ (X(\mu)_{[1/p']})'}) $ defined by all the functions $S_j$ from $X(\mu)_{[1/p]}$ to $ (X(\mu)_{[1/p']})'$ that are given by the expression
$$
S_j(g)(f):= \int j g h \, d\mu, \quad h \in X(\mu)_{[1/p']},
$$
for each $j \in X(\mu)'$. The operator norm is considered for this space. A direct identification shows that this is a description of the space $(X(\mu)_{[1/p]})^{\mathcal{V} \odot c_L}$ for this particular case.

An application of the product duality formula of Theorem \ref{prop1} gives that
$$
X(\mu)' = (X(\mu)_{[1/p]}
\cdot X(\mu)_{[1/p']})'= ( X(\mu)_{[1/p]} {\otimes}_{\pi_c} X(\mu)_{[1/p']})^\mathcal{V}
$$
$$
=(X(\mu)_{[1/p]})^{\mathcal{V} \odot c_L} = \mathcal{M}(X(\mu)_{[1/p]},(X(\mu)_{[1/p']})') \subseteq X(\mu)_{[1/p]})^{(X(\mu)_{[1/p']})'}.
$$
Therefore, we have found that for each $1 \le p$, the space of multiplication operators defined by elements of $X(\mu)'$, $\mathcal{M}(X(\mu)_{[1/p]},{ (X(\mu)_{[1/p']})'}) $ coincides in fact with $X(\mu)'$ isometrically.

\end{example}

\begin{example} \label{exnuv}
Other example that provides some information on duality in vector valued function spaces is the one given by 
the tensor product bilinear map $c=\otimes$ from $L^p(\mu) \times E$ on $L^p(\mu,E)$, ---the space of
 Bochner $p$-integrable functions--- for $1 < p< \infty$, and being $\mu$ a finite measure. Suppose also that $E^*$ has the Radon-Nikod\'ym property with respect to $\mu$, and take $\nu$ as the dual of the space $L^p(\mu,E)$, that is given by $L^{p'}(\mu,E^*)$; that is, $X= \mathbb R$. We have that
$$
L^p(\mu) {\otimes}_{\pi_c} E = L^p(\mu) {\otimes}_{\pi} E.
$$
On the other hand, 
$$
(L^p(\mu))^{\nu \odot  c_L} =  \Big\{ S_g:L^p(\mu) \to E^*: \, \textit{there is} \,g \in L^{p'}(\mu,E^*) \, 
$$
$$ 
\textit{such that} \, S_g(f)(x)= \int f \langle x, g(w) \rangle \, d \mu(w), \, f \in L^p(\mu), \, x \in E \Big\}.
$$
By the product duality formula, we have that 
$$
\Big( L^p(\mu) {\otimes}_{\pi} E \Big)^* \supseteq  \Big( L^p(\mu) {\otimes}_{\pi} E \Big)^\mathcal{V}  = (L^p(\mu))^{\nu \odot  c_L} \subseteq (L^p(\mu))^{E^*}.
$$
 This equality means that the dual of the Bochner space $L^p(\mu,E)$ generates a subspace of linear and continuous functions $L^p(\mu) \to E^*$ that is isometrically isomorphic to a space included in the dual of the projective tensor product of $L^p(\mu)$ and $E$. Moreover, for a function $g \in L^{p'}(\mu,E^*)$,
$$
\|\varphi_g\|=\|S_g\|_{(L^{p})^{E^*}}= \sup_{f \in B_{L^p(\mu)}, \, x \in B_E} \Big| \int f \langle x, g(w) \rangle \, d\mu(w) \Big|=
\sup_{x \in B_E} \| \langle x, g \rangle \|_{L^{p'}(\mu)},
$$
that is the $p'$-Pettis norm for the function $g$.

\end{example}

\subsection{Multiplication operators on Banach function spaces}
Let us show now other  application of our theorem that  proves a well-known formula of the theory of Banach lattices of functions regarding the authentic space of multiplication operators. It can be essentially found in \cite{caldelsan,delsan}; the same result with different notation is given in \cite{schepmultip}, and more examples of the duality formula in the case of multiplication operators can be found in \cite{SPfactomulti}.
Consider a pair of saturated Banach function spaces $X(\mu)$ and
$Y(\mu)$ over a finite measure $\mu$ (in the sense that has been explained in Example \ref{exBFS})
such that $X(\mu) \subseteq Y(\mu)'$, where $Y(\mu)'$ denotes the K\"othe dual of $Y(\mu)$. Then the
$\pi$-product space $X(\mu) \pi Y(\mu)$ can be defined as in \cite{delsan} and is again a
Banach function space over $\mu$ with $X(\mu) \pi Y(\mu) \subseteq
L^1(\mu)$. The definition of this $\pi$-product and its norm  is similar to the one of ``product" that we give here. However,  note that in  \cite{delsan} the product norm in defined for infinite sums, although in other versions only finite decompositions are considered (see \cite{schepmultip}). Completeness of the resulting product space is the advantage of considering infinite sums; this is not relevant here, since our product formula concerns the dual of the space, that is the same for the normed space and for its completion. In order to avoid confusion, we write $X(\mu) \pi_0 Y(\mu)$ for the normed space of linear combinations of single product of functions with the norm computed by means of finite sums.

 In Proposition 2.2(ii) of \cite{delsan}, the following product duality formula is proved.
$$
(X(\mu) \pi Y(\mu))'=X(\mu)^{Y(\mu)'}.
$$

 Consider the bilinear map $c:X(\mu) \times Y(\mu) \to
X(\mu) \pi_0 Y(\mu)$ given by the pointwise product $c(f,g):= f g$. Let us define $X= \mathbb R$ and $\mathcal V=
(X(\mu)\pi_0 Y(\mu))^{L^1(\mu)}$. A simple computation just taking into account the definition of the elements of each space shows that it coincides with $(X(\mu)\pi Y(\mu))'$, the K\"othe dual of the true $\pi$-product space. We take $G= X(\mu) \pi_0 Y(\mu) $ that coincides with $X(\mu) {\otimes}_{\pi_c} Y(\mu)$ by the construction of this space.  Then Theorem \ref{prop1} proves the product duality formula given above,
$$
(X(\mu) \pi Y(\mu))'=(X(\mu) {\otimes}_{\pi_c} Y(\mu))^{\mathcal V} =X(\mu)^{{\mathcal V} \odot c}=
X(\mu)^{Y(\mu)'}.
$$

\subsection{Multipliers on Banach spaces of analytic functions}

Let us show now other application of the duality formula given in a completely different context. We follow the ideas published in \cite{blaspav} (see also the references therein). Let $\mathcal{S}$ denote the space of all formal series $f= \sum_{j=0}^\infty \hat{f}(j) z^j$ with complex valued coefficients. We endow this space with the topology given by the seminorms $p_j(f):=\hat{f}(j)$, $j=1,...,\infty$. If $g$ is other series like this, the Hadamard product $\ast$ is defined as
$$
f \ast g= \sum_{j=0}^\infty \hat{f}(j) \hat{g}(j) z^j.
$$
Take a pair of spaces $X$ and $Y$ of analytic functions on the unit disk $\mathbb D$ such that each of them contains the polynomials and $X,Y \subseteq \mathcal{S}$ with continuous inclusions; following the definition given in the paper quoted, we say that the spaces are $\mathcal{S}$-admissible. We consider now two relevant spaces.

\begin{itemize}

\item The space that plays the role of product; as in the previous case and following our construction, we use finite sums in the following definitions of the space and the norm. In the original paper, infinite sums are considered, and then the corresponding space is complete.
We consider the space $X \bar\otimes Y$ defined by functions satisfying that they can be written as $f = \sum_{n=0}^k g_n \ast h_n,$ with $g_n \in X,$ $h_n \in Y$ and $\sum_{n=0}^k \|g_n\| \|h_n\| < \infty;$ the infimum in this expression for all decompositions of $f$ gives the natural norm for the space. We use the symbol $\bar\otimes$ instead of $\otimes$ used by the authors by the aim of clarity, since in our case only finite sums are considered and so we cannot expect completeness.  We show in what follows that the formula also works for finite sums decompositions, as is written when the product duality formula is involved.

\item The space that plays the role of a multiplication operators space. A series
$\lambda \in \mathcal S$ is said to be a (coefficient) \textit{multiplier} from $X$ to $Y$
 if $\lambda \ast f \in  Y$ for
each $f \in X$.
We denote the set of all multipliers from $X$ to $Y$ by $(X, Y )$ and define
$$
\|\lambda\|_{(X,Y)}:=
\sup \{ \| \lambda \ast f\|_Y: \|f\|_X \le 1 \}.
$$

\end{itemize}

The following result is proved in \cite{blaspav} (Theorem 2.3). Let $X, Y, Z$ be $\mathcal{S}$-admissible Banach spaces. Then
$$
(X \otimes Y,Z) = (X, (Y,Z)).
$$

It can be understood  again as a particular case of our product duality formula. The bilinear map $\ast:X \times Y \to X \bar\otimes Y$ plays the role of $c$, $G:=X \bar\otimes Y$ and  $\mathcal V:= (X \bar\otimes Y,Z)$, the space of multipliers from the product to $Z$. Thus, $(X {\otimes}_{\pi_c} Y)^\mathcal{V}=(X \bar\otimes Y,Z)$ and $X^{\mathcal V \odot c} :=(X,(Y,Z))$. Writing our product duality formula, we obtain
$$
(X \bar\otimes Y,Z)=(X {\otimes}_{\pi_c} Y)^\mathcal{V}=X^{\mathcal V \odot c} =(X,(Y,Z)),
$$
a version of the result that was proved in \cite{blaspav} for the norm defined for infinite sums.

\subsection{The product duality formula for the case of spaces of molecules associated to the linearization of Lipschitz bi-forms.}

Let us show the product duality formula for the case of the spaces of molecules, that appear in the standard techniques for linearizing Lipschitz operators. An operator $T$ from a metric space $(A,d_A)$ on a Banach space $Z$ is said to be Lipschitz if there is a constant $K >0$ such that for any pair of points $x_1, x_2 \in A$, we have
$$
\| T(x^1)-T(x^2)\| \le K d(x^1,x^2).
$$
The additional requirement $T(0)=0$ for a distinguished point $0 \in A$  is also assumed.
The space of molecules $\AE_A$ for the metric space $(A,d_A)$ is given by  the linear span of all the functions $A \to \mathbb R$ that can be written as differences of characteristic functions of each point, i.e.
$$
m_{x^1,x^2}(w):= \chi_{\{x^1\}}(w) - \chi_{\{x^2\}}(w), \quad w,x^1, x^2 \in A.
$$
The norm for this space is given by the formula
$$
L_A(z):= \inf \sum_{i=1}^n |\lambda_i| d(x^1_i,x^2_i), \quad \lambda_i \in \mathbb R,
$$
where the infimum is computed over all representations of $z$ as
$$
z=\sum_{i=1}^n \lambda_i (\chi_{\{x^1_i\}}(w) - \chi_{\{x^2_i\}}(w)).
$$
 A Lipschitz map $T:A \to Z$ always satisfies  a factorization scheme through the space of molecules as
$$
\xymatrix{
A \ar[rr]^{T} \ar@{->}[dr]_{j} & & Z. \\
& \AE_A \ar[ur]_{T_L} & }
$$
where $j$ is the Lipschitz isometry $j(x):=m_{x,0}$ for $x \in A$ and $T_L$ is the linearization of $T$ through the space of molecules.

Let us consider the notion of Lipschitz bi-form defined as follows. If $(A,d_A)$ and $(D,d_D)$ are metric spaces, we say that  a map $T:A \times D \to \mathbb R$ is a Lipschitz bi-form  if there is a constant $K >0$ such that for each $x^1,x^2 \in A$ and $y^1, y^2 \in D$,
$$
| T(x^1,y^1)-T(x^1,y^2) - T(x^2,y^1)+T(x^2,y^2) | \le K d_A(x^1,x^2) \cdot d_D(y^1,y^2).
$$
Additionally, $T(x,0)=T(0,y)=0$ for all $x \in A$ and $y \in D.$
If $(m_{x^1,x^2},m_{y^1,y^2}) \in \AE_A \times \AE_D$, we define the bilinearization  $T_B$ of $T$ as
 $$
T_B(m_{x^1,x^2},m_{y^1,y^2}) := T(x^1,y^1)-T(x^1,y^2) - T(x^2,y^1)+T(x^2,y^2),
$$
in  the way that alllows to $T$ to be factored through $T_B$ in the natural manner,
$$
\xymatrix{
A \times D \ar[rr]^{T} \ar@{->}[dr]_{j \times j} & & \mathbb R. \\
& \AE_A \times \AE_D \ar[ur]_{T_B} & }
$$
 We write $\mathcal B(\AE_A \times \AE_D)$ for the space of all the Lipschitz bi-forms, which can be bilinearized in this way.

Since $T_B$ can also be factored through the tensor product $\AE_A \otimes_\pi \AE_D$, the following construction makes sense. Take $G:=\AE_A \odot_\pi \AE_D$, the space defined by functions $u:A \times D \to \mathbb R$ that are linear combinations of pointwise products of functions of $\AE_A$ and $\AE_D$. Define a norm for this space as 
$$
\|u\|_\odot:= \inf \big\{ \sum_{i=1}^n \|f_i\|_{\AE_A} \|g_i\|_{\AE_D} : {u= \sum_{i=1}^n f_i \cdot g_i} \big\},
$$
where the infimum is computed over all suitable decompositions of $u$ as the one written in the formula. Clearly, an isometric isomorphism can be defined between this space and the tensor product $\AE_A \otimes_\pi \AE_D$.
Note that this norm can be computed also with the formula 
$$
\pi_c(u):=
\inf \big\{\sum_{i=1}^n |\lambda_i| d_A(x_i^1,x_i^2) d_D(y_i^1,y_i^2): \, {u= \sum_{i=1}^n \lambda_i  \, m_{x_i^1,x_i^2} \cdot m_{y_i^1,y_i^2}} \big\}.
$$

Take  the Banach space $X$ in our abstract construction of the duality as $X:= \mathbb R$, and consider the  bilinear map $c:\AE_A \times \AE_D \to \AE_A \odot_\pi \AE_D$ given by the pointwise product $c(f,g):=f  g$. Take also $\mathcal V= (\AE_A \otimes_\pi \AE_D)^*.$ Then, the $\pi_c$ norm is  the usual projective norm,  that coincides with $\| \cdot \|_\odot$.
Therefore, the product space is 
$$
\AE_A {\otimes}_{\pi_c} \AE_D := {\Big\{ u= \sum_{i=1}^n \lambda_i  m_{x_i^1,x_i^2} \cdot m_{y_i^1,y_i^2}:  \, \lambda_i \in \mathbb R, \, \ x_i^1,x_i^2 \in A, \, y_i^1,y_i^2 \in D \Big\} }.
$$
 Using
Example  \ref{ex4.2}  for the duality of the projective tensor product $(\AE_A \otimes_\pi \AE_D)^*= (\AE_A)^{\AE_D^*}$, we obtain that this class $\mathcal B(\AE_A \times \AE_D)$ of maps from the Cartesian product of metric spaces can be identified with the real maps that factor through  $\AE_A {\otimes}_{\pi_c} \AE_D.$ The product duality gives the equality with all the linear and continuous maps from $\AE_A$ to $\AE_D^*,$ that is,
$$
\mathcal B(\AE_A \times \AE_D) = (\AE_A)^{\AE_D^*}.
$$

\subsection{The product induced by the integration map in spaces of $p$-integrable functions with respect to a vector measure.}

Let us give first some necessary definitions on vector measure integration and the corresponding spaces of functions. If $(\Omega, \Sigma)$ is a measurable space,
let $m: \Sigma \to X$ be a Banach space valued countably additive
vector measure. Its semivariation is defined by
$\|m\|(A):=\sup_{x^* \in B_{Y^*}} | \langle m, x^* \rangle|(A)$,
$A \in \Sigma$, where $\langle m,x^* \rangle$ is the scalar 
measure given by $\langle m,x^* \rangle(A):= \langle m(A),x^*
\rangle$. Then  there exists $x^*\in
X^*$ such that  $m(A) = 0$ whenever $|\langle
m(A),x^*\rangle|= 0$, which implies that is equivalent to $m$ (same null sets). Such a measure  $\langle m, x^*\rangle$ defined by the composition of $m$ with a functional of $X^*$ is
called a Rybakov measure for $m$; there always exists at least one (\cite[Ch.IX]{dies}).
If $1 \le p < \infty$, a (scalar) measurable
function $f$ is said to be $p$-integrable with respect to $m$ if
$|f|^p$ is integrable with respect to all measures $|\langle m,
x^* \rangle|$ and for each $A \in \Sigma$ there exists an element
$\int_A |f|^p dm \in X$ such that $\langle \int_A |f|^p dm, x^*
\rangle =\int_A |f|^p d \langle m, x^* \rangle$, $x^* \in X^*$
(see \cite[Ch.3]{libro}).
The space $L^p(m)$ is defined by all the
equivalence classes (with respect to any Rybakov measure) of measurable real functions defined
on $\Omega$ that are $p$-integrable with respect to $m$. If $p=\infty$, the space $L^\infty(m)$ is defined as the space of bounded $1$-integrable functions with respect to $m$, that coincides with $L^\infty(\mu)$ for any Rybakov measure $\mu$ for $m$; the natural $L^\infty$-norm is considered for the space. If $1 \le p < \infty$, $L^p(m)$ is a $p$-convex order continuous 
Banach function space over any fixed Rybakov measure for $m$ when the
a.e. order and the norm 
$$
\|f\|_{L^p(m)}:= \Big( \sup_{x^* \in B_{X^*}} \int_\Omega |f|^p d
|\langle m, x^* \rangle| \Big)^{1/p}, \quad f \in L^p(m),
$$
are considered (see \cite[Proposition 5]{illi}, 
\cite{fernandez-mayoral-naranjo-saez-sanchez perez} and
\cite[Ch.3]{libro}). It must be pointed out that  $fg\in L^1(m)$ for any $f\in L^p(m)$ and $g\in L^{p'}(m)$, $1=1/p+1/p'$, and
for each $f \in L^p(m)$
$$
\|f\|_{L^p(m)}= \sup_{g \in B_{L_{p'}(m)}} \| \int_\Omega fg \,
dm\|.
$$
Thus, $L^p(m)$ is defined as the $1/p$-th power of $L^1(m)$.

Consider the bilinear map $c: L^p(m) \times
L^{p'}(m) \to L^1(m)$ given by the pointwise product  $c(f,g):= f g$ and the class
$\mathcal V$ of operators $L^1(m) \to X$ given by $T(h):= \int h
h_0 \, dm$ for $h_0 \in L^\infty(m)$. Then it can be easily seen
that
$$
\pi_c(h)= \|h\|_{L^1(m)}, \quad h \in L^1(m),
$$
as in the case of $p$-th powers of Banach function spaces, 
and $E {\otimes}_{\pi_c} F =L^1(m)$. We have that
$$
(E {\otimes}_{\pi_c} F)^\mathcal{V}  =
$$
$$
\big\{v_{I_{m_{h_0}}} \, | \, h_0 \in L^\infty(m), \,
v_{I_{m_{h_0}}}(f,g)= \int h_0fg \, dm \in X, \, f \in L^p(m), \,
g \in L^{p'}(m)  \big\}
$$
and
$$
\|v_{I_{m_{h_0}}}\|= \sup_{f \in B_{L^p(m)}, \, g \in
B_{L^{p'}(m)}} \| \int h_0 f g \, d m \|= \|h_0\|_{L^\infty(m)}.
$$
 On
the other hand,
$$
L^p(m)^{{{\mathcal V} \odot c}} =
$$
$$
\big\{ S_{I_{m_{h_0}}}: L^p(m) \to \mathcal{L}(L^{p'}(m),X) \,|
\, S_{I_{m_{h_0}}}(f)(\cdot)= \int h_0 f (\cdot) \, dm, \, h_0 \in
L^\infty(m) \big\}
$$
and
$$
\|S_{I_{m_{h_0}}} \|= \sup_{f \in B_{L^p(m)}} \big( \sup_{g \in
B_{L^{p'}(m)}} \| \int h_0 f g \, d m \|
\Big)=\|v_{I_{m_{h_0}}}\|= \|h_0\|_{L^\infty(m)}.
$$
Consequently
$$
\big( L^p(m) {\otimes}_{\pi_c} L^{p'}(m)
\big)^{\mathcal{V}}=
L^p(m)^{{{\mathcal V} \odot c}}=L^\infty(m).
$$
This formula concerns the vector measure duality between spaces of $p$-integrable functions, that was first studied in \cite{illi,proc} (see also \cite{irene,ferod,RuSaTMNA} and the references therein). Roughly speaking, in particular it asserts that the ``vector dual" space of $L^1(m)$ of the vector measure $m$ ---i.e, the dual space that appears when the duality  is defined by the bilinear operator induced by the integration map---, is always $L^\infty(m)$, since the usual dual space of $L^1(m)$ does not coincide  with this space in the general case.

\subsection{An application to spaces of integrable functions with respect to a vector measure.}

Let
$m:\Sigma \to Z$ be a Banach space valued countably additive
vector measure and let $1<p<\infty$. Consider the space of
$p$-integrable functions with respect to a vector measure
$L^p(m)$ and take it as the space $E$ of our general setting. Consider for each element $z^* \in Z^*$ the scalar measure
$\langle m, z^* \rangle(A):=\langle m(A), z^* \rangle$, $A \in
\Sigma$. Fix a Rybakov measure $\mu=|\langle m, z_0^*\rangle|$ and
consider the linear space of the Radon-Nikodym derivatives
$$
\mathcal{RN}(m) := \big\{ h= \frac{d \langle m, z^* \rangle}{d \mu},
\, z^* \in Z^* \big\} \subseteq (L^1(m))'
$$
endowed with the norm of $(L^1(m))'$; take it as $F$.

In recent years, some effort has been made in order to find how the set $\mathcal{RN}(m)$ can be used to find a description of the dual space of the spaces $L^p(m)$. In \cite{irene,proc}, such a description has been found, although some requirements are needed; the main application of this representation is to give a useful description of the weak topology in this space, what was finally done in \cite{ferod} (see also \cite{RuSaTMNA}).

In what follows we show how we can use the product construction to clarify this problem. Let $G:=  (L^{p'}(m))'$ and consider the Cartesian product $L^p(m) \times
\mathcal{RN}(m)$ and the bilinear map
$$
c:L^p(m) \times \mathcal{RN}(m) \to (L^{p'}(m))'
$$
given by the pointwise product $c(f,h)= f  h \in
(L^{p'}(m))'$. It is well defined, since for all $g \in L^{p'}(m)$
and $f \in L^p(m)$, $fg \in L^1(m)$ and so
$$
\int f h g \, d \mu= \int (fg) h \, d\mu \le \|fg\|_{L^1(m)}
\|h\|_{(L^1(m))'}
$$
$$
\le \|f\|_{L^p(m)} \|g\|_{L^{p'}(m)} \|h\|_{(L^1(m))'}, \quad h
\in \mathcal{RN}(m).
$$

For all the elements $w$ in the linear space generated by
$c(L^p(m) \times \mathcal{RN}(m))$, define
$$
{\pi_c}(w) := \inf \sum_{i=1}^n \|f_i\|_{L^p(m)}
\|h_i\|_{(L^1(m))'},
$$
where the infimum is computed over all decompositions of $w$ as
$\sum_{i=1}^n f_i h_i$, $f_i \in L^p(m)$ and $h_i \in  \mathcal{RN}(m)) \subseteq  (L^1(m))'$. Take $G$ as the space of all these $w$ with the norm $\pi_c$.

Consider now the space $L^{p'}(m)$ and take $\mathcal{V}$ as the space of real functionals from $(L^{p'}(m))'$ defined by the elements of $L^{p'}(m)$; we have taken $X:= \mathbb R.$

Let us write what is obtained by applying the product duality formula.
$$
\big( L^p(m) {\otimes}_{\pi_c} \mathcal{RN}(m)
\big)^{\mathcal{V}}
=
L^p(m)^{{{\mathcal V} \odot c}}
$$
$$
=\big\{ S_h: L^p(m) \to \mathcal{RM}(m)^* \,| \, h \in L^{p'}(m),
\, S_h(f)(\cdot):=\int hf \,\cdot \, d\mu \, \big\}
$$
$$
= L^{p'}(m) \subseteq   L^p(m)^{\mathcal{RN}(m)^*},
$$
where $L^{p'}(m)$ is endowed with its own norm. Therefore, by using the basic properties of the norm of the spaces $L^p(m)$ and the corresponding spaces of multiplication operators that can be found in \cite[Ch.3]{libro} it can be easily seen that  the first equality in
$$
L^{p'}(m) = \big( L^p(m) {\otimes}_{\pi_c} \mathcal{RN}(m)
\big)^{\mathcal{V}}  \subseteq \big( L^p(m) {\otimes}_{\pi_c} \mathcal{RN}(m)
\big)^{\mathbb{R}}
$$
holds isometrically. In other words, $L^p(m) {\otimes}_{\pi_c} \mathcal{RN}(m)$ is a sort of predual space of $L^{p'}(m)$, in the sense that this space can be identified with the linear functionals $\phi_h$ from $L^p(m) {\otimes}_{\pi_c} \mathcal{RN}(m)$ that factors as
$$
\xymatrix{ L^p(m) {\otimes}_{\pi_c} \mathcal{RN}(m)
\ar[rr]^{\phi_h} \ar@{->}[dr]_{i} & & \mathbb R. \\
&  (L^{p'}(m))' \ar[ur]_{\int h \cdot d \mu} & }
$$
for $h \in L^{p'}(m).$

\vspace{1cm}
\textit{The author wants to thank Professor O. Blasco by some comments that allowed to develop the example that is shown in Section 3.3, and to the anonymous referee for several suggestions, including Example \ref{exnuv} and the title of the paper.}

\bibliographystyle{amsplain}

\end{document}